\documentclass[10pt,reqno,oneside]{amsart}
\usepackage{stmaryrd}
\usepackage{amsmath}
\usepackage{amssymb}
\usepackage{amstext}
\usepackage{indentfirst}
\usepackage{cite}
\usepackage{mathrsfs}
\usepackage{enumerate}
\usepackage{hyperref}
\usepackage{amsfonts}
\usepackage{enumitem}
\usepackage{marvosym}
\usepackage{color}
\usepackage{pifont}
\theoremstyle{definition}
\newtheorem{Thm}{Theorem}
\newtheorem{Prop}{Proposition}

\newtheorem{Lemma}{Lemma}

\newtheorem{Rmk}{Remark}

\newtheorem{Conj}{Conjecture}

\numberwithin{equation}{section}
\title{The rigity of spaces with cyclic parallel Ricci tensor}

\author{
 Fengyuan Zhang \\
  Academy of mathematics and Systems Science\\
  \texttt{zhangfengyuan@amss.ac.cn} \\
  }

\begin{document}
\maketitle
\begin{abstract}
The aim of this paper is to classify some special Riemannian manifolds with cyclic parallel Ricci tensor, i.e.
    \begin{equation}
        D_{ijk}=\nabla_{i}R_{jk}+\nabla_{j}R_{ki}+\nabla_{k}R_{ij}=0\nonumber\\
    \end{equation}
These structures include non-compact gradient shrinking Ricci soliton, compact $(m > 1)$-quasi-Einstein manifolds with boundary and critical spaces. We will construct some integral identities and  make use of the curvature conditions reasonably to prove that the Ricci tensor is parallel.
\end{abstract}


\section{Introduction}
The study of generalizations of Einstein metrics has long been a central theme in differential geometry. Among these, three classes of geometric structures have attracted particular attention due to their deep connections with Ricci flow, warped product Einstein metrics, and general relativity: gradient Ricci solitons, quasi-Einstein manifolds, and vacuum static spaces.
\subsection{Three types of space}
\subsubsection{Gradient Ricci soliton}
A complete Riemannian manifold $(M^n, g)$ is called a  gradient Ricci soliton if there exists a smooth function $f$ on $M^n$ such that
    \begin{equation}\label{Eqn1}
        \nabla^2f+Ric=\lambda g
    \end{equation}
for some constant $\lambda\in {\mathbb R}$. Here, $\nabla ^2 f$ denotes the Hessian of $f$. Gradient Ricci solitons arise often as singularity models of the Ricci flow, and that is why understanding them is an important question in the field. Depending on the behavior of the Ricci flow on solitons, they are said to be  shrinking, or  steady, or   expanding if $\lambda>0$, or $\lambda=0$, or $\lambda<0$, respectively. Shrinkers can also be regarded as critical points of the Perelman’s entropy functional and play a significant role in Perelman’s resolution of the Poincaré conjecture\cite{perelman2002entropy}. The study of solitons has become increasingly important in both the study of the Ricci flow and metric measure space. This paper focuses on the classification results for shrinkers.

Hamilton \cite{hamilton1995formation} classified all $2$-dimensional gradient shrinking Ricci solitons, showing that they are either isometric to the Euclidean plane $\mathbb{R}^2$ or a quotient of the round sphere $\mathbb{S}^2.$ In dimension three, a combination of results due to Ivey \cite{ivey1994new}, Perelman \cite{perelman2003ricci}, Naber \cite{naber2010noncompact}, Ni–Wallach \cite{ni2008classification}, and Cao–Chen–Zhu \cite{cao2007recent} yields a complete classification: every complete $3$-dimensional gradient shrinking Ricci soliton is isometric to a finite quotient of one of the following model spaces: the round sphere $\mathbb{S}^3,$ the Gaussian shrinking soliton $\mathbb{R}^3,$ or the round cylinder $\mathbb{S}^{2}\times\mathbb{R}.$ Further classification results have been obtained under additional curvature assumptions. Locally conformally flat complete gradient shrinking Ricci solitons  were classified in \cite{cao2011locally} as finite quotients of ${\mathbb R}^n,{\mathbb S}^{n-k}\times{\mathbb R}^{k},{\mathbb S}^n.$ Subsequently, Fernández-López and García-Río\cite{fernandez2011rigidity}, and independently Munteanu and Sesum\cite{munteanu2013gradient}, established the classification of shrinking gradient Ricci solitons with harmonic Weyl curvature in the compact and non-compact settings, respectively. Later, in 2013, Cao and Chen\cite{cao2013} gave a complete classification of Bach-flat shrinking gradient Ricci solitons.
\subsubsection{Quasi-Einstein manifold}
An $m$-quasi-Einstein manifold is defined by the existence of a positive function 
$u$ satisfying the following system
\begin{align}\label{Eqn2}
    \left\{\begin{array} {l l} {\nabla^{2} u=\frac{u} {m} ( R i c-\lambda g )} & {\mathrm{in} \ M ,} \\ {u > 0} & {\mathrm{on} \ i n t ( M ) ,} \\ {u=0} & {\mathrm{on} \ \partial M ,} \\ \end{array} \right.
\end{align}
for some constant $\lambda\in {\mathbb R}$ and $0<m<\infty$. These structures are intimately related to the construction of warped product Einstein metrics (see \cite{case2011rigidity}) and have been extensively studied. 

Inspired by some successful works on gradient Ricci solitons, locally conformally flat quasi-Einstein manifolds were investigated in \cite{catino2010},\cite{he2012}. Chen and He \cite{chen2013bach} proved that a Bach-flat $(n \geq 4)$-dimensional compact quasi-Einstein manifold $(M,g)$ with smooth boundary $\partial M$ is either Einstein or a finite quotient of a warped product with $(n-1)$-dimensional Einstein fiber. Later, He, Petersen and Wylie\cite{he2014warped} showed that any quasi-Einstein manifold with harmonic curvature is necessarily rigid. Diógenes and Gadelha \cite{diogenes2022remarks} initiated the study of quasi-Einstein metrics with zero radial Weyl curvature. Recently, Cao, Li \cite{cao2025quasi} classified $m$-quasi-Einstein manifolds with harmonic Weyl curvature.

\subsubsection{Critical Spaces}
A third class, we consider a family of geometric structures unified by the following equation:
\begin{equation}\label{Eqn3}
        \nabla^2f+\frac{R}{n(n-1)}fg=(f+a)\mathring{Ric}+bg,
    \end{equation}
    where $f$ is a smooth function on $(\textit{M},g),$ $a,b\in \mathbb{R}$ are constants, $R$ denotes the scalar curvature, $\mathring{Ric}$ denotes the traceless Ricci tensor, and $g$ is the metric. Equation \eqref{Eqn3} can be seen as a generalization of a family of well-known and extensively studied geometric structures. We refer to this class as critical spaces. Recall that Fischer and Marsden \cite{fischer1975deformations} established  the $L^2-$formal adjoint of the linearization of the scalar curvature is given by
          $$\mathcal{L}_{ g}^{*}(f)=\nabla^2 f-\Delta fg-fRic.$$
    We now introduce the following three classes of critical spaces via this operator.
\begin{itemize}
  \item CPE metrics: $(a, b) = (1, 0)$.
\end{itemize}

    Equation \eqref{Eqn3} then reduces to the following form:
    $$
    \mathcal{L}_{ g}^{*}(f)=\mathring{Ric}
    $$
    This is precisely the critical point equation (CPE) for the total scalar curvature functional restricted to the space of constant scalar curvature metrics. More precisely, a CPE metric is a triple 
    $(M,g,f)$ on a compact oriented manifold with constant scalar curvature, where $f$ is a smooth potential function satisfying the Euler-Lagrange equation of the Hilbert-Einstein action on 
    \begin{equation}
        \mathcal{C}=\left\{g| \operatorname{Vol}_{\textit{M}^{n}}(g)=1 \text { and } R_g=\mathrm{constant}\right\}.\nonumber
    \end{equation}
It is a long-standing conjecture, attributed to Besse\cite{besse2007einstein}:
\begin{Conj}\label{conj1}
    CPE metrics are Einstein.
\end{Conj}
The proof of the Conjecture \ref{conj1} has been a subject of interest for many authors.
In \cite{lafontaine1983geometrie}, Lafontaine showed that the Conjecture \ref{conj1} is true for a locally conformally
flat manifold. Later on, this result was improved by Chang, Hwang and Yun\cite{yun2014total} under
harmonic curvature assumption which is clearly weaker than locally conformally
flat condition considered in Lafontaine’s result. This same authors in \cite{chang2012critical}, were able to solve the conjecture for a manifold satisfying the parallel Ricci tensor condition.

\begin{itemize}
  \item Vacuum static space: $(a, b) = (0, 0)$.
\end{itemize}

When $a=b=0$, equation \eqref{Eqn3} reduces to vacuum static equation
\begin{equation}\label{static}
     \nabla^2f+\frac{R}{n(n-1)}fg=f\mathring{Ric},\nonumber
\end{equation}
which arises naturally as the kernel of the $L^2$
 -formal adjoint of the linearized scalar curvature operator $\mathcal{L}_{ g}^{*}(f)$ and also corresponds to static solutions of the Einstein field equations in general relativity. Fischer-Marsden \cite{fischer1975deformations} made the following conjecture:
\begin{Conj}\label{conj2}
    Any compact vacuum static space is Einstein.
\end{Conj}

 If it is true, by Obata’s theorem \cite{obata1962certain}, such a space must be a standard sphere
or a Ricci flat space. However, it turns out the conjecture is not true. In dimension $n\geq3$, Kobayashi\cite{kobayashi1982differential} and Lafontaine\cite{lafontaine1983geometrie} independently proved that closed locally conformally flat examples with scalar curvature $n(n-1)$ are isometric to the standard sphere $S^{n}(1),$  a finite quotient of $S^{1} ( \sqrt{1 / n} )\ \times S^{n-1} ( \sqrt{( n-2 ) / n} ),$ or a finite quotient of the warped product $S^{1} ( \sqrt{1 / n} ) \times_{h} S^{n-1} ( \sqrt{( n-2 ) / n} ),$ where $h$ is a positive periodic function on the circle factor. This provided the first counterexample to Conjecture \ref{conj2}.  Subsequently, Qing and Yuan\cite{qing2013note} extended this result to the Bach-flat setting, showing that closed Bach-flat vacuum static spaces with the same scalar curvature are either $S^{n}(1),$ a finite quotient of $S^{1} ( \sqrt{1 / n} ) \times F,$ or a finite quotient of $S^{1} ( \sqrt{1 / n} ) \times_{h} F,$ with $F$ a closed Einstein manifold of dimension $n-1$ and Einstein constant $n.$ In four dimensions, Kim and Shin\cite{kim2018four} classified such spaces under the assumption of harmonic curvature. In the case of higher dimensions, one can consult \cite{kim2026vacuum} or \cite{Li2021}. In recent years, Baltazar, Barros, Batista, and Viana \cite{baltazar2020static}, as well as Ye \cite{ye2023closed}, have
demonstrated that closed vacuum static spaces of dimension $n \geq 5$ with zero radial Weyl curvature are Bach-flat, thereby reducing their classification to the aforementioned theorem of Qing and Yuan.

\begin{itemize}
  \item Miao-Tam critical metric: $(a, b) = (0, -\frac{1}{n-1})$.
\end{itemize}

When $a=0$ and $b=-\frac{1}{n-1},$ This choice reduces the equation \eqref{Eqn3} to the so-called Miao-Tam critical metric
$$
\mathcal{L}_{ g}^{*}(f)=g,
$$
which has attracted considerable attention in recent years due to its deep connections with the volume functional and the static Einstein equation.

More precisely, a Miao-Tam critical metric is a compact Riemannian manifold 
$(\textit{M},g)$ with non-empty boundary 
$\partial\textit{M}$ that admits a non-constant solution $f$ of the equation. In analogy with the classical static case, the potential function $f$ is required to vanish precisely on the boundary and to be strictly positive in the interior of $\textit{M}.$ Such metrics were first investigated by Miao and Tam in \cite{Miao2009}, where they were characterized as critical points of the volume functional under suitable boundary constraints.

In 2011, Miao and Tam raised the question of whether there exist Miao-Tam critical metrics with non-constant sectional curvature on a compact manifold whose boundary is isometric to a standard round sphere (see \cite{Miao2011}). As a partial answer, they showed that any Einstein Miao-Tam critical metric $(M,g,f)$ must be isometric to a geodesic ball in a simply connected space form ${\mathbb R}^n,{\mathbb H}^n,$
or ${\mathbb S}^n$, without requiring any boundary restriction. Subsequently, Baltazar and Ribeiro Jr. in \cite{baltazar2018} extended this classification to the setting of parallel Ricci tensor, obtaining the same rigidity conclusion.

\subsection{Cyclic parallel Ricci tensor}
Before proceeding, it is necessary to recall the following 3-tensor defined in terms of the Ricci tensor:
\begin{equation}
D(X,Y,Z)=\nabla_{X}\operatorname{Ric}(Y,Z)+\nabla_{Y}\operatorname{Ric}(Z,X)+\nabla_{Z}\operatorname{Ric}(X,Y),\nonumber
\end{equation}
for any $X,Y,Z\in\Gamma(TM).$ For simplicity, we rewrite the above equation in index notation as
\begin{equation}
    D_{ijk}=\nabla_{i}R_{jk}+\nabla_{j}R_{ki}+\nabla_{k}R_{ij}.
\end{equation}
A Riemannian manifold is said to have cyclic parallel Ricci tensor if $D_{ijk}=0.$ Such concept was introduced by A. Gray in \cite{gray1978einstein} and clearly generalizes the definition of an Einstein manifold. Moreover, it is evident that a parallel Ricci tensor implies the cyclic parallel condition, while the converse is not true in general (see \cite{gray1978einstein} and \cite{jelonek2008tensors} for further details). 

While each of the three classes above has been studied under various curvature assumptions—including  locally conformally flat, harmonic Weyl, and Bach-flat, a systematic treatment under the unifying hypothesis cyclic Ricci parallel  has not yet been fully explored. In 2019, aiming to improve the result of Baltazar\cite{baltazar2018}, Sheng and Wang\cite{sheng2019critical} initiated the study of Miao–Tam critical metrics with cyclic parallel Ricci tensor. They proved that, under the assumption that the sectional curvature is non-positive or that the scalar curvature satisfies the following upper bound
    \begin{equation*}
        R\leq-\sqrt{\frac{(n-1)(n-2)}{2}}|W|-\sqrt{n(n-1)}|\mathring{Ric}|,
    \end{equation*}
where $W$ denotes the Weyl curvature tensor, the metric is necessarily Einstein. Consequently, the space must be isometric to a geodesic ball in a simply connected space form $\mathbb {R}^n$ or $\mathbb {H}^n.$ Shortly afterwards, Baltazar and Silva\cite{baltazar2021static} pioneered the rigidity analysis of more general critical spaces under the cyclic parallel Ricci assumption. They proved that Besse's conjecture \ref{conj1} holds in dimension three for such metrics. However, for dimensions $n\geq4$,they were unable to establish the rigidity theorem due to the difficulty of appropriately controlling the radial Weyl curvature term arising from the Bochner formula. In addition, they obtained a classification of three-dimensional Miao–Tam critical metrics under the cyclic parallel Ricci condition. For higher-dimensional Miao–Tam critical metrics, however, their result also requires the additional assumption of non-positive scalar curvature. Recently, Andrade and Baltazar \cite{andrade2026einstein} classified compact Einstein-type manifolds with cyclic parallel Ricci tensor. Unfortunately, their classification is also restricted to the three-dimensional case.

The aim of this paper is to fill this gap. Our main result shows that, in dimension $n\geq3$, the cyclic parallel condition $D=0$ actually forces the Ricci tensor to be parallel in the three geometric structures introduced above. In particular, for Miao–Tam critical metrics, a stronger conclusion can be obtained: the metric is necessarily Einstein. 

It is instructive to compare the cyclic parallel Ricci condition adopted in the present paper with another geometric assumption that has recently proven useful in the rigidity theory of vacuum static spaces, namely the existence of a non-Killing closed conformal vector field. In a very recent work, Ye\cite{Ye2025} studied closed vacuum static spaces admitting such a vector field $\xi$ and established a complete classification under the assumption that the scalar curvature is $n(n-1)$ and the Cotton tensor satisfies $C(\cdot,\xi,\cdot)=0.$ In that setting, the characteristic function of the conformal vector field serves as a solution to the vacuum static equation, and the rigidity result follows from an identity involving the Cotton tensor. Although both the cyclic parallel Ricci condition and the closed conformal vector field assumption lead to similar classification outcomes, they are conceptually distinct: the former is an intrinsic curvature condition imposed directly on the covariant derivative of the Ricci tensor, whereas the latter is an extrinsic symmetry-type condition on the metric. Consequently, our results are independent of and complementary to those obtained under the closed conformal vector field hypothesis.
\subsection{Main results}
We are now in a position to state our main results. In this paper, we always assume that the Riemannian manifold $(M,g)$ is complete and simply connected.

\begin{Thm}\label{thm1}
    Let $(M^{n},g,f)$ be a non-compact gradient shrinking Ricci soliton with cyclic parallel Ricci tensor. Then $(M^{n},g)$ is isometric to a finite quotient of $N^{n-k}\times\mathbb R^k,$ where $N^{n-k}$ is an $(n-k)-$dimensional Einstein maniflod. 
\end{Thm}

\begin{Thm}\label{thm2}
    Let $(M^{n},g,u)$ be a non-trivial, compact  $(m > 1)$-quasi-Einstein manifold with smooth boundary $\partial M$ and cyclic parallel Ricci tensor. Then $(M^{n},g)$ is isometric to  standard hemisphere $\mathbb S^{n}_{+}$ or a quotient of $\mathbb{S}^k_{+}\times N^{n-k},$ where $N^{n-k}$ is an $(n-k)-$dimensional closed Einstein maniflod. 
\end{Thm}

\begin{Thm}\label{thm3}
    Let $(M^n,g,f)$ be a non-trivial closed vacuum static space with cyclic parallel Ricci tensor. Then $(M^{n},g)$ is isometric to either the Euclidean sphere $\mathbb S^n$ or a quotient of $\mathbb{S}^k\times N^{n-k},$ 
    where $N^{n-k}$ is an $(n-k)-$dimensional closed Einstein maniflod.
\end{Thm}

\begin{Thm}\label{thm4}
     Let $(M^{n},g,f)$ be a non-trivial closed CPE metric with cyclic parallel Ricci tensor and scalar curvature $n(n-1).$ Then $(M^{n},g)$ is isometric to Euclidean sphere $\mathbb S^{n}(1).$
\end{Thm}

\begin{Thm}\label{thm5}
     Let $(M^n,g,f)$ be a compact, oriented, connected Miao-Tam critical metric with smooth boundary $\partial M$ and cyclic parallel Ricci tensor. Then $(M^{n},g)$ is isometric to a geodesic ball in a simply connected space form $\mathbb{R}^{n}$, $\mathbb{H}^{n}$ or $\mathbb{S}^{n}.$
\end{Thm}

Our proof method is mainly motivated by the integral identity firstly established in \cite{cao2007compact}, which holds on any closed gradient Ricci soliton:
\begin{equation}
    \int_{\textit{M} }|{\rm div}Rm|^2 e^{-f}=\int_{\textit{M} }|\nabla Ric|^2 e^{-f}.\nonumber
\end{equation}
Our strategy is to establish analogous integral identities for each of the three classes of spaces and to apply the curvature assumptions $D=0$ in a suitable manner, thereby proving that the covariant derivative of the Ricci tensor vanishes.
\subsection{Outline of the article}
    The structure of this article is organized as follows. In Sect. 2, we recall some fundamental notions in Riemannian geometry and collect several basic facts concerning the three Hessian-type equations. In Sect. 3, we establish the corresponding integral identities for the three spaces, which are then used to prove our main classification theorems.
\section{Preliminaries}
         In this section, we introduce some fundamental concepts in Riemannian geometry.
    \subsection{Notations and conventions}
        We start with the set-up of notations and reviewing some basic notions in Riemannian geometry. See   \cite{petersen2006riemannian} as references on Riemannian geometry.

         Let $(M^n, g)$ be a smooth Riemannian manifold of dimension $n\ge3$ with Levi-Civita connection $\nabla$ induced by the metric $g.$ Denote by $TM$ the tangent bundle of $M^n$ and by $\Gamma(TM)$ the space of smooth sections of $TM.$ The Riemann curvature tensor $Rm:\Gamma(TM)\times \Gamma(TM)\times \Gamma(TM)\rightarrow \Gamma(TM)$ is defined by
        $$R(X, Y) Z=\nabla_{X} \nabla_{Y} Z-\nabla_{Y} \nabla_{X} Z-\nabla_{[X, Y]} Z,$$
        for any smooth vector fields $X, Y,Z\in\Gamma(TM).$

    At the origin point of the local normal coordinate system, we denote by
    \begin{equation}
        Rm=\{R_{ijkl}\}, \quad Ric=\{R_{ij}\}, \quad R\nonumber
    \end{equation}
    the Riemann curvature tensor, the Ricci tentor, the scalar curvature of the metric $g,$ respectively.
By the once contracted second Bianchi identiy, we can see the divergence of $Rm$
    \begin{equation}\label{divR}
        ({\rm div}Rm)_{ijk}=\nabla_{l}R_{ijkl}=\nabla_{j}R_{ik}-\nabla_{i}R_{jk}.
    \end{equation}
    The Schouten tensor $ {\rm A}$ of $g$ is defined by
        \begin{equation*}
              {\rm A}=Ric-\frac{ R}{2(n-1)}\,g.
        \end{equation*}
The Weyl tensor $\rm{W}$ is given by
                \begin{equation*}\label{decompofR}
                        W_{ijkl}=R_{i j k l} - \frac{1}{n-2}({\rm A}\owedge  g)_{ijkl},
                \end{equation*}
where the symbol $\owedge$ represents the Kulkarni-Nomizu product which is defined for any two symmetric $(0,2)-$tensors $U$ and $V$ as follows
        $$(U \owedge V)_{i j k l}=U_{ik}V_{jl}+U_{jl}V_{ik}-U_{il}V_{jk}-U_{jk}V_{il}.$$
Through the Schouten tensor, we can define the Cotton tensor $ {\rm C}$ as follows
                \begin{eqnarray}\label{Cotton}
                        C_{ijk}=\nabla_{i}A_{jk}-\nabla_{j}A_{ik}.\nonumber
                \end{eqnarray}

\subsection{Gradient Ricci solitons}
    Recall that a complete Riemannian manifold $(M^n, g)$ is called a  gradient Ricci soliton if there exists a smooth function $f$ on $M^n$ such that the Ricci tensor $Ric=\{R_{ij}\}$ of the metric $g=\{g_{ij}\}$ satisfies the equation $Ric +\nabla^2 f=\lambda g$  for some $\lambda \in{\mathbb R}$, or in local normal coordinates, 
    \begin{equation}\label{eqn1}
        R_{ij}+\nabla_i\nabla_jf=\lambda g_{ij}.
    \end{equation}
    Here, $\nabla ^2 f=\{\nabla_i\nabla_j f\}$ denotes the Hessian of $f$. The Ricci soliton $(M^n, g, f)$ is said to be  
     shrinking, or  steady, or   expanding if $\lambda>0$, or $\lambda=0$, or $\lambda<0$, respectively. We now collect several well-known differential identities for gradient Ricci soliton:
     \begin{equation}\label{div Ric}
         \nabla_{j}R_{ij}=\frac{1}{2}\nabla_{i}R=R_{ij}\nabla_{j}f,
     \end{equation}
     and
     \begin{equation}\label{ridial1}
         R_{ijkl}\nabla_{l}f=\nabla_{j}R_{ik}-\nabla_{i}R_{jk}.
     \end{equation}
     Hence, by \eqref{divR}, \eqref{div Ric} and \eqref{ridial1} we have the following two identities
     \begin{equation}\label{2.5}
         \nabla_{l}(R_{ijkl}e^{-f})=0,
     \end{equation}
     and
     \begin{equation}\label{2.6}
         \nabla_{j}(R_{ij}e^{-f})=0.
     \end{equation}
\subsection{Quasi-Einstein manifolds}
    In this subsection, we recall basic facts on $m$-quasi-Einstein manifolds. First of all, we remember that the fundamental equation of an m-quasi- Einstein manifold with smooth boundary $\partial M$ in the tensorial language :
    \begin{equation}\label{eqn2}
        \nabla_{i}\nabla_{j}u=\frac{u}{m}(R_{ij}-\lambda g_{ij}),
    \end{equation}
    where $u>0$ in the interior of $M$ and $u=0$ on the boundary $\partial M$.
    For $m\neq1,$ He-Petersen-Wylie  \cite{he2012} defined the modified scalar curvature and Ricci tensor:
    \begin{equation}
        \rho(x):=\frac{(n-1)\lambda-R}{m-1}
        ~\text{and}~
        P:=Ric-\rho g.\nonumber
    \end{equation}
 By direct calculation, the next formula shows the trace of $P$
    \begin{equation}\label{trP}
        tr(P)=(n-1)\lambda-(m+n-1)\rho.
    \end{equation}
    Analogous to \eqref{div Ric} and \eqref{2.6}, we have
    \begin{equation}\label{13}
        P_{ij}\nabla_{j}u=\frac{u}{2}\nabla_{i}\rho,
    \end{equation}
    and
    \begin{equation}
        \nabla_{j}(u^{m+1}P_{ij})=0.
    \end{equation}
    Next, they define a new algebraic curvature tensor, which satisfies a similar divergence condition and also traces to a multiple of $P$
    \begin{equation*}
        Q:=Rm+\frac{1}{m}P\owedge g+\frac{\rho-\lambda}{2m}g\owedge g.
    \end{equation*}
     By directly tracing the index 1 and 3 of $Q$, we can obtain
    \begin{equation}
        tr(Q)=\frac{m+n-2}{m}P.\nonumber\\
    \end{equation}
    Moreover, the radial $Q$ curvature is given by
    \begin{Prop}\cite{he2012}
    For $(m\neq1)$-quasi-Einstein manifolds, the following holds:
       \begin{equation}\label{ridial2}
            Q_{ijkl}\nabla_{l}u=\frac{u}{m}(\nabla_{i}P_{jk}-\nabla_{j}P_{ik})-\frac{1}{m}(P_{jl}g_{ik}-P_{il}g_{jk})\nabla_{l}u.
       \end{equation}
    \end{Prop}
By \eqref{ridial2}, it present identity involving the divergence of $Q$ 
    \begin{equation}\label{16}
        \nabla_{l}(u^{m+1}Q_{ijkl})=0,
    \end{equation}
    which is similar to \eqref{2.5}.
    \begin{Rmk}
    The divergence formula above is equivalent to 
    \begin{equation}\label{17}
       \nabla_{l}Q_{ijkl}=-(m+1)u^{-1}Q_{ijkl}\nabla_{l}u ,
    \end{equation}
    which we shall adopt more frequently in the subsequent calculations.
    \end{Rmk}
 In addition, Case\cite{case2011rigidity} established the elliptic equation for the scalar curvature       
  \begin{equation}\label{18}
        \frac{1}{2}\Delta R+\frac{m+1}{2u}\langle\nabla R,\nabla u\rangle=(\lambda-\rho)tr(P)-|P|^2.
    \end{equation}
Finally, we can give the simplified differential identity of $P$ and $Q$ under the assumption that the scalar curvature is constant. Indeed, by \eqref{13},\eqref{ridial2} and \eqref{17}, we have
    \begin{equation}\label{divP}
        \nabla_{j}P_{ij}=P_{ij}\nabla_{j}u=0,
    \end{equation}
    and 
    \begin{align}\label{divQ}
      \nabla_{l}Q_{ijkl}
      &=-(m+1)u^{-1}Q_{ijkl}\nabla_{l}u\nonumber\\
      &=-\frac{m+1}{m}(\nabla_{i}P_{jk}-\nabla_{j}P_{ik}).
    \end{align}
Moreover, when the scalar curvature is constant, the elliptic equation \eqref{18} shows
    \begin{equation}\label{P2}
        |P|^2=(\lambda-\rho)tr(P).
    \end{equation}

\subsection{Vacuum static spaces and related critical spaces}
In this subsection, we study the general critical-type equation \eqref{Eqn3} and rewrite it in the form of components
    \begin{equation}\label{eqn3}
        \nabla_i\nabla_jf+\frac{R}{n(n-1)}fg_{ij}=(f+a)\mathring{R}_{ij}+bg_{ij}.
    \end{equation}
 Indeed, when $a=1,b=0$, it reduces to the well-known Critical Point Equation , whose solutions are called CPE metrics; when $b+\frac{ R}{n(n-1)}a=0,$ $f+a$ satisfies vacuum static equation; when $a=0$ and $b=-\frac{1}{n-1},$ it is known as Miao-Tam critical metric. 

Without loss of generality, we next consider the equation \eqref{eqn3}. By taking the trace in both sides of \eqref{eqn3}, we have
         \begin{equation}\label{Deltafab}
               \Delta f+\frac{ R}{n-1}f=nb.\nonumber
         \end{equation}
              As \cite[Proposition 2.3]{corvino2000scalar}, it can be proven that the scalar curvature of \((M^n,g)\) is constant when \(f\) is a non-constant solution of \eqref{eqn3} on \(M^n\).

Motivated by the construction of $(m\neq1)$-quasi-Einstein manifolds, we define an algebraic curvature-type tensor which is characterized by the property that traces to a multiple of $Ric$
\begin{equation*}
    S:=Rm+Ric\owedge g-\frac{R}{2(n-1)}g\owedge g.
\end{equation*}
By directly tracing the index 1 and 3 of $S,$ we can obtain
    \begin{equation}
        tr(S)=(n-1)Ric.\nonumber\\
    \end{equation}
\begin{Prop}\label{prop2}
Let $f$ be a smooth solution to \eqref{eqn3} defined on an $(n\ge3)-$dimensional Riemannian manifold, and let $S$ be defined as above. Then
    \begin{equation}\label{radial3}
        S_{ijkl}\nabla_{l}f=(f+a)(\nabla_{i}R_{jk}-\nabla_{j}R_{ik})+(R_{jl}g_{ik}-R_{il}g_{jk})\nabla_{l}f.\nonumber
    \end{equation}
\end{Prop}
\begin{proof}
Commuting the third derivatives of $f$ and using equation \eqref{eqn3}, we obtain
    \begin{align}
        R_{ijkl}\nabla_{l}f
        &=\nabla_{i}\nabla_{j}\nabla_{k}f-\nabla_{j}\nabla_{i}\nabla_{k}f\nonumber\\
        &=\nabla_{i}((f+a)\mathring{R}_{jk}-\frac{R}{n(n-1)}fg_{jk})-\nabla_{j}((f+a)\mathring{R}_{ik}-\frac{R}{n(n-1)}fg_{ik})\nonumber\\
        &=(f+a)(\nabla_{i}\mathring{R}_{jk}-\nabla_{j}\mathring{R}_{ik})+\nabla_{i}fR_{jk}-\nabla_{j}fR_{ik}\nonumber\\
        &\quad-\frac{R}{(n-1)}(g_{jk}\nabla_{i}f-g_{ik}\nabla_{j}f).\nonumber
    \end{align}
    By the definition of the curvature tensor $S,$ we get
    \begin{align}
      S_{ijkl}\nabla_{l}f
      &=R_{ijkl}\nabla_{l}f+(R_{ik}g_{jl}+R_{jl}g_{ik}-R_{il}g_{jk}-R_{jk}g_{il})\nabla_{l}f\nonumber\\
      &\quad-\frac{R}{(n-1)}(g_{ik}g_{jl}-g_{il}g_{jk})\nabla_{l}f\nonumber\\
      &=(f+a)(\nabla_{i}R_{jk}-\nabla_{j}R_{ik})+(R_{jl}g_{ik}-R_{il}g_{jk})\nabla_{l}f,\nonumber
    \end{align}    
    this is what we want to prove.
\end{proof}

Finally, we still provide the divergence formula for $S,$ but unfortunately, it does not hold true for ${\rm div}((f+a)^2S)=0$ like a  $(m\neq1)$-quasi-Einstein manifold. However, it is not a big deal and only requires some caution in certain areas. Based on Proposition \ref{prop2}, we verify that
\begin{align}\label{divS}
    \nabla_{l}S_{ijkl}(f+a)
    &=-2(\nabla_{i}R_{jk}-\nabla_{j}R_{ik})(f+a)\nonumber\\
    &=-2S_{ijkl}\nabla_{l}f-2(R_{jl}g_{ik}-R_{il}g_{jk})\nabla_{l}f.
\end{align}

\section{Proof The Main Results}
\subsection{The proof of Theorem \ref{thm1}}
To begin with, we present an integral identity for a general complete non-compact Ricci soliton. Although a similar identity can be found in \cite{munteanu2013gradient}, our alternative derivation, which is carried out with an eye toward the subsequent proof, leads to a slightly modified integral remainder term that better serves the uniformity of our later integral identities.
\label{sec:headings}
    \begin{Prop}\label{prop3}
    Let  $(M, g,f)$ be a complete non-compact gradient Ricci soliton, $\phi\in C_c^{\infty}(M)$ be a cut-off function on $M$. Then
      \begin{align}\label{26}
         \int_{\textit{M} }|{\rm div}Rm|^2 e^{-f}\phi^2
         &=\int_{\textit{M} }|\nabla Ric|^2 e^{-f}\phi^2-\int_{\textit{M} }R_{ijkl}R_{ik}\nabla_{l}f e^{-f}\nabla_{j}\phi^2\nonumber\\
         &\quad+\int_{\textit{M} }\nabla_{i}R_{jk}R_{ik} e^{-f}\nabla_{j}\phi^2.  
      \end{align}
    
    \end{Prop}
    
    \begin{proof}
    Initially, we invoke \eqref{divR} and \eqref{ridial1} to infer
       \begin{align}\label{27}
           \int_{\textit{M} }|{\rm div}Rm|^2 e^{-f}\phi^2
           &=\int_{\textit{M}}R_{ijkl}\nabla_{l}f(\nabla_{j}R_{ik}-\nabla_{i}R_{jk})e^{-f}\phi^2\nonumber\\
           &=2\int_{\textit{M}}R_{ijkl}\nabla_{l}f\nabla_{j}R_{ik}e^{-f}\phi^2\nonumber\\
           &=-2\int_{\textit{M}}R_{ik}\nabla_{j}(R_{ijkl}\nabla_{l}f e^{-f}\phi^2)\nonumber\\
           &=-2\int_{\textit{M}}R_{ijkl}(\lambda g_{jl}-R_{il})R_{ik}e^{-f}\phi^2-2\int_{\textit{M}}R_{ijkl}\nabla_{l}fR_{ik}\nabla_{j}\phi^2\nonumber\\
           &=2\int_{\textit{M}}R_{ijkl}R_{ik}R_{jl}e^{-f}\phi^2-2\lambda\int_{\textit{M}}|Ric|^2e^{-f}\phi^2\nonumber\\
           &\quad-2\int_{\textit{M}}R_{ijkl}\nabla_{l}fR_{ik}\nabla_{j}\phi^2,       
       \end{align} 
       which we use integration by parts in the third equality and equation\eqref{eqn1},\eqref{2.5} in the forth equality.
       
    On the oher hand, by using \eqref{divR} again one sees that
    \begin{align}\label{28}
        \int_{\textit{M}}|{\rm div}Rm|^2 e^{-f}\phi^2
        &=\int_{\textit{M}}|\nabla_{j}R_{ik}-\nabla_{i}R_{jk}|e^{-f}\phi^2\nonumber\\
        &=2\int_{\textit{M}}|\nabla Ric|^2e^{-f}\phi^2-2\int_{\textit{M}}\nabla_{j}R_{ik}\nabla_{i}R_{jk}e^{-f}\phi^2.
    \end{align}
    Using integration by parts again, we derive that 
    \begin{align}\label{29}
        -\int_{\textit{M}}\nabla_{j}R_{ik}\nabla_{i}R_{jk}e^{-f}\phi^2
        &=\int_{\textit{M}}R_{ik}\nabla_{j}(\nabla_{i}R_{jk}e^{-f}\phi^2)\nonumber\\
        &=\int_{\textit{M}}R_{ik}\nabla_{j}\nabla_{i}R_{jk}e^{-f}\phi^2-\int_{\textit{M}}\nabla_{i}R_{jk}R_{ik}\nabla_{j}fe^{-f}\phi^2\nonumber\\
        &\quad+\int_{\textit{M}}R_{ik}\nabla_{i}R_{jk}e^{-f}\nabla_{j}\phi^2.       
    \end{align}
    Proceeding, from the Ricci identities for any Riemannian manifold, it holds
    \begin{align}\label{30}
        \int_{\textit{M}}R_{ik}\nabla_{j}\nabla_{i}R_{jk}e^{-f}\phi^2
        &=\int_{\textit{M}}R_{ik}(\nabla_{i}\nabla_{j}R_{jk}+R_{lk}R_{ljij}+R_{jl}R_{lkij})e^{-f}\phi^2\nonumber\\
        &=-\int_{\textit{M}}R_{ik}\nabla_{j}R_{jk}e^{-f}\nabla_{i}\phi^2+\int_{\textit{M}}R_{ij}R_{ik}R_{jk}e^{-f}\phi^2\nonumber\\
        &\quad-\int_{\textit{M}}R_{ijkl}R_{ik}R_{jk}e^{-f}\phi^2.
    \end{align}
    For the second term on the right-hand side of \eqref{29}, the divergence theorem yields
    \begin{align}\label{31}
        -\int_{\textit{M}}\nabla_{i}R_{jk}R_{ik}\nabla_{j}fe^{-f}\phi^2
        &=\int_{\textit{M}}R_{ik}R_{jk}\nabla_{i}\nabla_{j}fe^{-f}\phi^2+\int_{\textit{M}}R_{ik}R_{jk}\nabla_{j}fe^{-f}\nabla_{i}\phi^2\nonumber\\
        &=\lambda\int_{\textit{M}}|Ric|^2e^{-f}\phi ^2-\int_{\textit{M}}R_{ij}R_{ik}R_{jk}e^{-f}\phi^2\nonumber\\
        &\quad+\int_{\textit{M}}R_{ik}R_{jk}\nabla_{j}fe^{-f}\nabla_{i}\phi^2.      
    \end{align}
    Now, substituting \eqref{29},\eqref{30} and \eqref{31} into \eqref{28} and combing with \eqref{div Ric}, we get
    \begin{align}\label{32}
        \int_{\textit{M}}|{\rm div}Rm|^2 e^{-f}\phi^2
        &=2\int_{\textit{M}}|\nabla Ric|^2e^{-f}\phi^2-2\int_{\textit{M}}R_{ijkl}R_{ik}R_{jk}e^{-f}\phi^2\nonumber\\
        &\quad+2\lambda\int_{\textit{M}}|Ric|^2e^{-f}\phi ^2+2\int_{\textit{M}}R_{ik}\nabla_{i}R_{jk}e^{-f}\nabla_{j}\phi^2.       
    \end{align}
    Finally, adding \eqref{27} and \eqref{32} completes the proof.   
    \end{proof}
    Based on this key identity, which is established in Proposition \ref{prop3}, and combined with the curvature condition, we can prove the main theorem \ref{thm1}.
    \begin{proof}[Proof of Theorem \ref{thm1}]
    
    A direct norm computation shows that
    \begin{equation}\label{3.8}
    |{\rm div}Rm|^2=2|\nabla Ric|^2-2\nabla_{i}R_{jk}\nabla_{j}R_{ik},
    \end{equation}
and
\begin{equation}\label{3.9}
    \frac{1}{3}|D|^2=D_{ijk}\nabla_{i}R_{jk}=|\nabla Ric|^2+2\nabla_{i}R_{jk}\nabla_{j}R_{ik}.
\end{equation}
Thus, we sum the expressions obtained in \eqref{3.8} and \eqref{3.9} to deduce a key property in order to obtain our results, namely
\begin{equation}
    |{\rm div}Rm|^2+ \frac{1}{3}|D|^2=3|\nabla Ric|^2\nonumber
\end{equation}
Suppose $D=0,$ then we have $|{\rm div}Rm|^2=3|\nabla Ric|^2.$ Combining this with Proposition \ref{prop3}, we obtain
\begin{equation}\label{3.10}
    2\int_{\textit{M} }|\nabla Ric|^2 e^{-f}\phi ^2=\int_{\textit{M} }\nabla_{i}R_{jk}R_{ik} e^{-f}\nabla_{j}\phi^2-\int_{\textit{M} }R_{ijkl}R_{ik}\nabla_{l}f e^{-f}\nabla_{j}\phi^2.
\end{equation}
Finally, we estimate the two integral residues respectively, the Cauchy-Schwarz inequality shows

\begin{align}
     |\int_{\textit{M} }\nabla_{i}R_{jk}R_{ik} e^{-f}\nabla_{j}\phi^2|
    &\leq2\int_{\textit{M} }|\nabla_{i}R_{jk}||R_{ik}| e^{-f}\phi|\nabla_{j}\phi| \nonumber\\ 
    &\leq\frac{1}{2}\int_{\textit{M} }|\nabla Ric|^2 e^{-f}\phi^2+2\int_{\textit{M} }|Ric|^2e^{-f}|\nabla\phi|^2,\nonumber
\end{align}
and

\begin{align}
    |\int_{\textit{M} }R_{ijkl}R_{ik}\nabla_{l}f e^{-f}\nabla_{j}\phi^2|
    &\leq2\int_{\textit{M} }|R_{ijkl}\nabla_{l}f||R_{ik}| e^{-f}\phi|\nabla_{j}\phi|  \nonumber\\
    &\leq\frac{1}{6}\int_{\textit{M} }|{\rm div}(Rm)|^2 e^{-f}\phi^2+6\int_{\textit{M} }|Ric|^2e^{-f}|\nabla\phi|^2\nonumber\\
    &=\frac{1}{2}\int_{\textit{M} }|\nabla Ric|^2 e^{-f}\phi^2+6\int_{\textit{M} }|Ric|^2e^{-f}|\nabla\phi|^2,\nonumber
\end{align}
 we have used the fact that $|{\rm div}Rm|^2=3|\nabla Ric|^2$ again in the last equality. Consequently, \eqref{3.10} implies that
\begin{align}
    \int_{\textit{M} }|\nabla Ric|^2 e^{-f}\phi^2\leq8\int_{\textit{M} }|Ric|^2e^{-f}|\nabla\phi|^2.\nonumber
\end{align}
Take $\phi$ such that $\phi=1$ on $B_{p}(r),$ $\phi=0$ on $M \backslash B_{p}(2 r),$ and $|\nabla\phi|\leq\frac{c}{r},$ it follows that as $r\rightarrow\infty$
    \begin{equation}
        \int_{B_{p}(r)}|\nabla Ric|^2 e^{-f}\leq\frac{8c}{r^2}\int_{B_{p}(2r)\backslash B_{p}(r)}|Ric|^2e^{-f}\rightarrow0,\nonumber
    \end{equation}
where the last convergence is justified by Theorem 1.1 of \cite{munteanu2013gradient}, which guarantees that, for any shrinker, the following holds:
\begin{equation}
    \int_{\textit{M} }|Ric|^2e^{-f}<\infty.\nonumber
\end{equation}
This proves the Ricci tensor is parallel.

 Applying the result of Petersen-Wylie \cite[Proposition 1.3]{petersen2009rigidity} to this situation, we conclude that any gradient shrinking Ricci soliton with parallel Ricci tensor is rigid. This proves the theorem.
\end{proof}

\subsection{The proof of Theorem \ref{thm2}}
    
In view of the fact that the computations on$(m > 1)$-quasi-Einstein manifold are considerably more intricate than those on the Ricci soliton, it is expedient to isolate a preliminary lemma prior to the derivation of our integral identity. This auxiliary result will greatly facilitate the ensuing analysis by allowing us to apply it directly whenever needed.

    \begin{Lemma}\label{lemm1}
    Let $(M, g,u)$ be a nontrivial compact simply connected $(m > 1)$-quasi-Einstein manifold with smooth boundary $\partial M$ and constant sacalar curvature. Then
    \begin{align}\label{lemma1}
        \int_{\textit{M}}\nabla_{i}P_{jk}P_{ik}\nabla_{j}uu^{m-1}=-\frac{1}{m}\int_{\textit{M}}tr(P^3)u^{m}-\frac{1}{m}\int_{\textit{M}}(\rho-\lambda)|P|^2u^{m}.
    \end{align}
    \end{Lemma}
    \begin{proof}
    By using equation \eqref{eqn2} and \eqref{divP}, one obtains that
      \begin{align}
          \int_{\textit{M}}\nabla_{i}P_{jk}P_{ik}\nabla_{j}uu^{m-1}
          &=-\int_{\textit{M}}P_{jk}\nabla_{i}(P_{ik}\nabla_{j}uu^{m-1})\nonumber\\
          &=-\int_{\textit{M}}P_{jk}P_{ik}\nabla_{i}\nabla_{j}uu^{m-1}\nonumber\\
          &=\frac{1}{m}\int_{\textit{M}}tr(P^3)u^{m}-\frac{1}{m}\int_{\textit{M}}(\rho-\lambda)|P|^2u^{m}.\nonumber
      \end{align} 
In the first equality we have applied integration by parts. Here we observe for the first time the necessity of the condition $m>1,$ which ensures the vanishing of the boundary terms.
    \end{proof}
We now present the integral identity on $(m > 1)$-quasi-Einstein manifold. Unlike in the case of the Ricci soliton, here we need to assume that the scalar curvature is constant. This assumption considerably simplifies our computations and is also natural under our curvature condition that the Ricci tensor is cyclic parallel.    
    \begin{Prop}\label{prop4}
        Let $(M, g,u,\lambda)$ be a non-trivial compact simply connected $(m > 1)$-quasi-Einstein manifold with smooth boundary $\partial M$ and constant sacalar curvature. Then
        \begin{equation}
            \int_{\textit{M}}|{\rm div}Q|^2 u^{m}=\frac{(m+1)^2}{m^2}\int_{\textit{M}}|\nabla P|^2 u^{m}.
        \end{equation}
    \end{Prop}
    \begin{proof}

   On one hand, we invoke \eqref{divQ} to infer
    \begin{align}\label{39}
        \int_{\textit{M}}|{\rm div}Q|^2 u^{m}
        &=-\frac{(m+1)^2}{m}\int_{\textit{M}}Q_{ijkl}\nabla_{l}u(\nabla_{i}P_{jk}-\nabla_{j}P_{ik})u^{m-1}\nonumber\\
        &=-\frac{2(m+1)^2}{m}\int_{\textit{M}}Q_{ijkl}\nabla_{l}u\nabla_{j}P_{ik}u^{m-1}\nonumber\\
        &=\frac{2(m+1)^2}{m}\int_{\textit{M}}P_{ik}\nabla_{j}(Q_{ijkl}\nabla_{l}uu^{m+1}u^{-2})\nonumber\\
        &=\frac{2(m+1)^2}{m}\int_{\textit{M}}Q_{ijkl}P_{ik}\nabla_{j}\nabla_{l}uu^{m-1}\nonumber\\
        &\quad-\frac{4(m+1)^2}{m^2}\int_{\textit{M}}Q_{ijkl}P_{ik}\nabla_{j}u\nabla_{l}uu^{m-2}.\nonumber\\
    \end{align}
In the fourth equality we have applied integration by parts together with \eqref{16}. Here we observe once again that the conditiont $m>1$ is necessary. Next, we compute the two terms on the right-hand side of \eqref{39} separately
    \begin{align}
        \int_{\textit{M}}Q_{ijkl}P_{ik}\nabla_{j}\nabla_{l}uu^{m-1}
        &=\int_{\textit{M}}\frac{u}{m}Q_{ijkl}P_{ik}(P_{jl}+(\rho-\lambda)g_{jl})u^{m-1}\nonumber\\
        &=\frac{1}{m}\int_{\textit{M}}Q_{ijkl}P_{ik}P_{jl}u^{m}+\frac{1}{m}\int_{\textit{M}}(\rho-\lambda)\frac{m+n-2}{m}|P|^2u^{m},
    \end{align}
    and
    \begin{align}\label{41}
        \int_{\textit{M}}Q_{ijkl}P_{ik}\nabla_{j}u\nabla_{l}uu^{m-2}
        &=\int_{\textit{M}}\frac{u}{m}(\nabla_{i}P_{jk}-\nabla_{j}P_{ik})P_{ik}\nabla_{j}uu^{m-2}\nonumber\\
        &=\frac{1}{m}\int_{\textit{M}}\nabla_{i}P_{jk}P_{ik}\nabla_{j}u u^{m-1},
    \end{align}
where we used the fact that when the scalar curvature $R$ is constant, then $P(\nabla u)=0$ and $|P|^2=(\lambda-\rho)tr(P)$ is also constant. The term on the right-hand side of \eqref{41} has already been handled in Lemma \ref{lemm1}. Hence
    \begin{align}\label{42}
       \int_{\textit{M}}|{\rm div}Q|^2 u^{m}
        &=\frac{2(m+1)^2}{m^2}\int_{\textit{M}}Q_{ijkl}P_{ik}P_{jl}u^{m}+\frac{4(m+1)^2}{m^3}\int_{\textit{M}}P_{ij}P_{ik}P_{jk}u^{m}\nonumber\\
        &+\frac{2(m+1)^2(m+n)}{m^3}\int_{\textit{M}}(\rho-\lambda)|P|^2u^{m}.
    \end{align}
Furthermore, a second application of \eqref{divQ} gives
    \begin{align}\label{43}
        \int_{\textit{M}}|{\rm div}Q|^2 u^{m}
        &=\frac{(m+1)^2}{m^2}\int_{\textit{M}}|\nabla_{i}P_{jk}-\nabla_{j}P_{ik}|^2u^{m}\nonumber\\
        &=\frac{2(m+1)^2}{m^2}\{\int_{\textit{M}}|\nabla P|^2u^{m}-\int_{\textit{M}}\nabla_{j}P_{ik}\nabla_{i}P_{jk}u^{m}\},
    \end{align}
where
    \begin{align}\label{44}
       - \int_{\textit{M}}\nabla_{j}P_{ik}\nabla_{i}P_{jk}u^{m}
       &= \int_{\textit{M}}P_{ik}\nabla_{j}(\nabla_{i}P_{jk}u^{m})\nonumber\\
       &=\int_{\textit{M}}P_{ik}\nabla_{j}\nabla_{i}P_{jk}u^{m}+m
       \int_{\textit{M}}\nabla_{i}P_{jk}P_{ik}\nabla_{j}uu^{m-1}.
    \end{align}
For the first term on the left-hand side of \eqref{44}, we apply the Ricci identity and express the resulting Ricci and Riemann curvature tensors in terms of the tensors $P$ and $Q.$ Indeed, we have
    \begin{align}\label{45}
       \int_{\textit{M}}P_{ik}\nabla_{j}\nabla_{i}P_{jk}u^{m}
       &= \int_{\textit{M}}P_{ik}(\nabla_{i}\nabla_{j}P_{jk}+P_{lk}R_{ljij}+P_{jl}R_{lkij})u^{m}\nonumber\\
       &=\int_{\textit{M}}P_{ik}P_{lk}(P_{il}+\rho g_{il})u^{m}\nonumber\\
       &\quad-\int_{\textit{M}}P_{ik}P_{jl}\{
       Q_{ijkl}-\frac{1}{m}(P_{ik}g_{jl}+P_{jl}g_{ik}-P_{il}g_{jk}-P_{jk}g_{il})\nonumber\\
       &\quad-\frac{\rho-\lambda}{m}(g_{ik}g_{jl}-g_{il}g_{jk})\}u^{m}\nonumber\\
       &=-\int_{\textit{M}}Q_{ijkl}P_{ik}P_{jl}u^{m}+\frac{m-2}{m}\int_{\textit{M}}P_{ij}P_{ik}P_{jk}u^{m}\nonumber\\
       &\quad+(\rho+\frac{2}{m}trP-\frac{\rho-\lambda}{m})\int_{\textit{M}}|P|^2u^{m}+\frac{\rho-\lambda}{m}\int_{\textit{M}}|tr P|^2u^{m}.
    \end{align}
    The computation for the second term on the left-hand side of \eqref{44} has already been carried out in Lemma \ref{lemm1}. 
    Substituting \eqref{44}, \eqref{45}, and \eqref{lemma1} into \eqref{43}, and then combining the resulting identity with \eqref{P2}, we obtain
    \begin{align}\label{46}
        \int_{\textit{M}}|{\rm div}Q|^2 u^{m}
        &=\frac{2(m+1)^2}{m^2}\int_{\textit{M}}|\nabla P|^2 u^{m}-\frac{2(m+1)^2}{m^2}\int_{\textit{M}}Q_{ijkl}P_{ik}P_{jl}u^{m}\nonumber\\
        &\quad-\frac{4(m+1)^2}{m^3}\int_{\textit{M}}P_{ij}P_{ik}P_{jk}u^{m}\nonumber\\
        &\quad-\frac{2(m+1)^2(m+n)}{m^3}\int_{\textit{M}}(\rho-\lambda)|P|^2u^{m}.
    \end{align}
Finally, adding \eqref{42} and \eqref{46} we complete the proof.

     \end{proof}
With the aid of Proposition \ref{prop4}, the proof of Theorem \ref{thm2} follows by an argument analogous to that of Theorem \ref{thm1}.
\begin{proof}[Proof of Theorem \ref{thm2}]

We begin with a straightforward computation
   \begin{align}\label{3.21}
    |{\rm div}Q|^2
    &=\frac{(m+1)^2}{m^2}|\nabla_{i}P_{jk}-\nabla_{j}P_{ik}|^2\nonumber\\
    &=\frac{2(m+1)^2}{m^2}(|\nabla P|^2-\nabla_{i}P_{jk}\nabla_{j}P_{ik}).
    \end{align}
Adding \eqref{3.9} and \eqref{3.21} yields
    \begin{equation}
        \frac{(m+1)^2}{m^2}|{\rm div} Q|^2+\frac{1}{3}|D|^2=3|\nabla P|^2.\nonumber
    \end{equation}
Under the cyclic parallel condition $D=0$, the above identity reduces to $\frac{(m+1)^2}{m^2}|{\rm div}Q|^2=3|\nabla P|^2.$ Finally, together with Proposition \ref{prop4}, this  implies  $\nabla P=0,$ i.e. the Ricci tensor is parallel.

By the result of He, Petersen and Wylie \cite[Corollary 1.12]{he2014warped} on quasi-Einstein manifolds with harmonic curvature, we assert that $(M,g,u)$ is rigid. Together with their classification result\cite[Proposition 2.3 and 2.5]{he2014warped}, this completes the proof.

\end{proof}

\subsection{The rigity of critical spaces}
   
Consider now the vacuum static equation , this case in fact corresponds to the $m=1$ setting in the quasi-Einstein context, though it is substantially more difficult to handle. In order to unify our results, we shall now consider the general critical equation \eqref{eqn3}.
    \begin{Lemma}\label{lemm2}
    Let $(M^n, g,f)$ be an $(n\ge3)-$dimensional closed Riemannian manifold satisfying \eqref{eqn3}. Then
        \begin{align}
        \int_{\textit{M}}|\mathring{Ric}|^2(f+a)=0.\nonumber
        \end{align}

    \end{Lemma}
    \begin{proof}
     Multiplying both sides of equation \eqref{eqn3} by $\mathring{R}_{ij}$ and integrating over $\textit{M},$ we apply integration by parts to the left-hand side, thereby obtaining the claimed result.
    \end{proof}

Therefore, for the sake of consistency, we always substitute $\int_{\textit{M}}|\mathring{Ric}|^2(f+a)^2$ for $\int_{\textit{M}}|\mathring{Ric}|^2f(f+a)$ in the following discussion. Unlike in the $(m > 1)$-quasi-Einstein case, the term $|\mathring{Ric}(\nabla f)|^2$ can no longer be neglected here. Indeed, we have the following lemma.

    \begin{Lemma}\label{lemm3}
    Let $(M^n, g,f)$ be an $(n\geq3)-$dimensional closed Riemannian manifold satisfying \eqref{eqn3}. Then
         \begin{align}
        \int_{\textit{M}}\nabla_{i}\mathring{R}_{jk}\mathring{R}_{ik}\nabla_{j}f(f+a)
        &=-\int_{\textit{M}}tr(\mathring{Ric}^3)(f+a)^2+\frac{1}{n(n-1)}\int_{\textit{M}}R|\mathring{Ric}|^2(f+a)^2\nonumber\\
        &\quad-\int_{\textit{M}}|\mathring{Ric}(\nabla f)|^2.
    \end{align}
    \end{Lemma}
    \begin{proof}
    Integrating by parts on the left-hand and combining equation \eqref{eqn3}, we immediately get
        \begin{align}
        \int_{\textit{M}}\nabla_{i}\mathring{R}_{jk}\mathring{R}_{ik}\nabla_{j}f(f+a)
        &=-\int_{\textit{M}}\mathring{R}_{jk}\nabla_{i}(\mathring{R}_{ik}\nabla_{j}f(f+a))\nonumber\\
        &=-\int_{\textit{M}}\mathring{R}_{jk}\mathring{R}_{ik}\nabla_{i}\nabla_{j}f(f+a)-\int_{\textit{M}}|\mathring{Ric}(\nabla f)|^2\nonumber\\
        &=-\int_{\textit{M}}tr(\mathring{Ric}^3)(f+a)^2+\frac{1}{n(n-1)}\int_{\textit{M}}R|\mathring{Ric}|^2(f+a)^2\nonumber\\
        &\quad-\int_{\textit{M}}|\mathring{Ric}(\nabla f)|^2.
        \end{align} 
    In the last equality we have applied Lemma \ref{lemm2} twice: first to obtain$\int_{\textit{M}}|\mathring{Ric}|^2f(f+a)=\int_{\textit{M}}|\mathring{Ric}|^2(f+a)^2,$ and second to ensure the vanishing of $b\int_{\textit{M}}|\mathring{Ric}|^2(f+a).$  
    \end{proof}
   Finally, we present a lemma that will be useful in handling the integral of the Ricci curvature along $\nabla f.$

    \begin{Lemma}\label{lemm4}
    Let $(M^n, g,f)$ be an $(n\geq3)-$dimensional closed Riemannian manifold satisfying \eqref{eqn3}. Then
        \begin{align}\label{lemma4}
            \int_{\textit{M}}\mathring{Ric}(\nabla f,\nabla f)=-\int_{\textit{M}}|\mathring{Ric}|^2(f+a)^2.
        \end{align}
    \end{Lemma}
    \begin{proof}
         Multiplying both sides of equation \eqref{eqn3} by $\mathring{R}_{ij}(f+a)$ and integrating over $\textit{M},$ we apply integration by parts to the left-hand side, which yields the desired result.
    \end{proof}
We are now in a position to give an analogous integral identity on the critical space. Unlike in the case of $(m > 1)$-quasi-Einstein manifold, here we multiply $(f+a)^2$ as the auxiliary function to ensure that the principal term is nonnegative. Moreover, the inner product term cannot be handled at this stage.
    
  \begin{Prop}\label{prop5}
      Let $(M^n, g)$ be an $(n\ge3)-$dimensional closed Riemannian manifold. Suppose that $f$ is a non-constant smooth solution to \eqref{Eqn3}. Then
          \begin{align}
           \int_{\textit{M}}|{\rm div}S|^2(f+a)^2    &=4\int_{\textit{M}}|\nabla\mathring{Ric}|^2(f+a)^2+8\int_{\textit{M}}\nabla_{i}\mathring{R}_{jk}\mathring{R}_{ik}\nabla_{j}f(f+a)\nonumber\\
           &\quad+2\int_{\textit{M}}\langle \nabla|\mathring{Ric}|^2,\nabla f  \rangle(f+a).  
          \end{align}
  \end{Prop}  
 \begin{proof}
Firstly, by virtue of \eqref{divS}, we obtain

    \begin{align}\label{53}
        \int_{\textit{M}}|{\rm div}S|^2 (f+a)^2
        &=-4\int_{\textit{M}}S_{ijkl}\nabla_{l} f(\nabla_{j}\mathring{R}_{ik}-\nabla_{i}\mathring{R}_{jk})(f+a)\nonumber\\
        &=-8\int_{\textit{M}}S_{ijkl}\nabla_{l} f\nabla_{j}\mathring{R}_{ik}(f+a)\nonumber\\
        &=8\int_{\textit{M}}\mathring{R}_{ik}\nabla_{j}(S_{ijkl}\nabla_{l} f(f+a))\nonumber\\
        &=8\int_{\textit{M}}\mathring{R}_{ik}\{-2S_{ijkl}\nabla_{j} f+2(R_{lj}g_{ik}-R_{jk}g_{il})\nabla_{j}f\}\nabla_{l}f\nonumber\\
        &\quad+8\int_{\textit{M}}\mathring{R}_{ik}S_{ijkl}\nabla_{j}\nabla_{l}f(f+a) 
        +8\int_{\textit{M}}\mathring{R}_{ik}S_{ijkl}\nabla_{j}f\nabla_{l}f \nonumber\\
        &=8\int_{\textit{M}}\mathring{R}_{ik}S_{ijkl}\nabla_{j}\nabla_{l}f(f+a)\nonumber-8\int_{\textit{M}}\mathring{R}_{ik}S_{ijkl}\nabla_{j}f\nabla_{l}f
        \\
        &\quad-16\int_{\textit{M}}|\mathring{Ric}(\nabla f)|^2-\frac{16}{n}\int_{\textit{M}}R\mathring{Ric}(\nabla f,\nabla f).
    \end{align}
Next, combining equation \eqref{eqn3} and Proposition \ref{prop2}, we arrive at the following two identities
    \begin{align}\label{54}
        \int_{\textit{M}}\mathring{R}_{ik}S_{ijkl}\nabla_{j}\nabla_{l}f(f+a)
        &=\int_{\textit{M}}\mathring{R}_{ik}S_{ijkl}\{(f+a)\mathring{R}_{jl}-\frac{R}{n(n-1)}fg_{jl}+bg_{jl}\}(f+a)\nonumber\\
        &=\int_{\textit{M}}S_{ijkl}\mathring{R}_{ik}\mathring{R}_{jl}(f+a)^2-\frac{1}{n}\int_{\textit{M}}R|\mathring{Ric}|^2(f+a)^2,
    \end{align}
and
    \begin{align}\label{55}
        \int_{\textit{M}}\mathring{R}_{ik}S_{ijkl}\nabla_{j}f\nabla_{l}f
        &=\int_{\textit{M}}\mathring{R}_{ik}\nabla_{j}f\{(f+a)(\nabla_{i}\mathring{R}_{jk}-\nabla_{j}\mathring{R}_{ik})+
        (R_{jl} g_{ik}-R_{il} g_{jk})\nabla_{l} f\}\nonumber\\
        &=\int_{\textit{M}}\nabla_{i}\mathring{R}_{jk}\mathring{R}_{ik}\nabla_{j}f(f+a)-\frac{1}{2}\int_{\textit{M}}\langle \nabla|\mathring{Ric}|^2,\nabla f  \rangle(f+a)\nonumber\\
        &\quad-\int_{\textit{M}}|\mathring{Ric}(\nabla f)|^2-\frac{1}{n}\int_{\textit{M}}R\mathring{Ric}(\nabla f,\nabla f).      
    \end{align}
Substituting \eqref{54},\eqref{55} and \eqref{lemma4} into \eqref{53}, we obtain
    \begin{align}\label{56}
     \int_{\textit{M}}|{\rm div} S|^2 (f+a)^2
     &=8\int_{\textit{M}}S_{ijkl}\mathring{R}_{ik}\mathring{R}_{jl}(f+a)^2+8\int_{\textit{M}}tr(\mathring{Ric}^3)(f+a)^2\nonumber\\
     &\quad-\frac{8}{n(n-1)}\int_{\textit{M}}R|\mathring{Ric}|^2(f+a)^2+4\int_{\textit{M}}\langle\nabla|\mathring{Ric}|^2,\nabla f  \rangle(f+a).
    \end{align}
On the other hand, by using \eqref{divS} again 
    \begin{align}\label{57}
      \int_{\textit{M}}|{\rm div} S|^2 (f+a)^2
      &=4\int_{\textit{M}}|\nabla_{i}\mathring{R}_{jk}-\nabla_{j}\mathring{R}_{ik}|^2(f+a)^2\nonumber\\
      &=8\int_{\textit{M}}|\nabla\mathring{Ric}|^2(f+a)^2-8\int_{\textit{M}}\nabla_{i}\mathring{R}_{jk}\nabla_{j}\mathring{R}_{ik}(f+a)^2.
    \end{align}
Integrating the last term by parts gives
    \begin{align}\label{58}
     -\int_{\textit{M}}\nabla_{i}\mathring{R}_{jk}\nabla_{j}\mathring{R}_{ik}(f+a)^2
     &=\int_{\textit{M}}\mathring{R}_{ik}\nabla_{j}(\nabla_{i}\mathring{R}_{jk}(f+a)^2)\nonumber\\
     &=\int_{\textit{M}}\mathring{R}_{ik}\nabla_{j}\nabla_{i}\mathring{R}_{jk}(f+a)^2+2\int_{\textit{M}}\nabla_{i}\mathring{R}_{jk}\mathring{R}_{ik}\nabla_{j}f(f+a).
    \end{align}
Applying the Ricci identity to the first term on the right-hand side of \eqref{58}, and expressing the resulting Ricci and Riemann curvature tensors in terms of the traceless Ricci tensor and the tensor $S$, we obtain
    \begin{align}\label{59}
        \int_{\textit{M}}\mathring{R}_{ik}\nabla_{j}\nabla_{i}\mathring{R}_{jk}(f+a)^2
        &=\int_{\textit{M}}\mathring{R}_{ik}(\nabla_{i}\nabla_{j}\mathring{R}_{jk}+\mathring{R}_{lk}R_{ljij}+\mathring{R}_{jl}R_{lkij})(f+a)^2\nonumber\\
        &=\int_{\textit{M}}\mathring{R}_{ik}\mathring{R}_{lk}(\mathring{R}_{il}+\frac{R}{n}g_{il})(f+a)^2\nonumber\\
        &\quad-\int_{\textit{M}}\mathring{R}_{ik}\mathring{R}_{jl}\{S_{ijkl}-(R_{ik}g_{jl}+R_{jl}g_{ik}-R_{il}g_{jk}-R_{jk}g_{il})\nonumber\\
        &\quad+\frac{R}{n-1}(g_{ik}g_{jl}-g_{il}g_{jk})\}(f+a)^2\nonumber\\
        &=-\int_{\textit{M}}S_{ijkl}\mathring{R}_{ik}\mathring{R}_{jl}(f+a)^2-\int_{\textit{M}}tr(\mathring{Ric}^3)(f+a)^2\nonumber\\
        &\quad+\frac{1}{n(n-1)}\int_{\textit{M}}R|\mathring{Ric}|^2(f+a)^2.
    \end{align}
Substituting \eqref{58} and \eqref{59} into \eqref{57}, and then applying Lemma \ref{lemm3}, we arrive at
    \begin{align}\label{60}
      \int_{\textit{M}}|{\rm div} S|^2 (f+a)^2
      &=8\int_{\textit{M}}|\nabla\mathring{Ric}|^2(f+a)^2-8\int_{\textit{M}}S_{ijkl}\mathring{R}_{ik}\mathring{R}_{jl}(f+a)^2\nonumber\\
      &\quad-24\int_{\textit{M}}tr(\mathring{Ric}^3)(f+a)^2
      +\frac{24}{n(n-1)}\int_{\textit{M}}R|\mathring{Ric}|^2(f+a)^2\nonumber\\
      &\quad-16\int_{\textit{M}}|\mathring{Ric}(\nabla f)|^2.  
    \end{align}
Finally, adding \eqref{56} and \eqref{60} and applying Lemma \ref{lemm3} again, we complete the proof.
    \end{proof}
    Although this integral identity is not as perfect as one might wish, we can handle the inner product term by making more refined use of the curvature assumptions.

    \begin{proof}[Proof of Theorem \ref{thm3} and \ref{thm4}]

Recall the definition of 3-tensor $D.$ A direct computation of the inner product gives
    \begin{align}
        D_{ijk}\mathring{R}_{ik}\nabla_{j}f=\frac{1}{2}\langle \nabla|\mathring{Ric}|^2,\nabla f  \rangle+2\nabla_{i}\mathring{R}_{jk}\mathring{R}_{ik}\nabla_{j}f.\nonumber
    \end{align}
Under the cyclic parallel assumption $D_{ijk}=0,$ the above identity reduces to
    \begin{align}
        \langle \nabla|\mathring{Ric}|^2,\nabla f  \rangle=-4\nabla_{i}\mathring{R}_{jk}\mathring{R}_{ik}\nabla_{j}f.\nonumber
    \end{align}
At this point, Proposition \ref{prop5} asserts that
    
    \begin{align}
        \int_{\textit{M}}|{\rm div} S|^2 (f+a)^2=4\int_{\textit{M}}|\nabla\mathring{Ric}|^2(f+a)^2.\nonumber
    \end{align}
    The rest of the proof follows exactly the same argument as that of Theorem \ref{thm2}, the only additional observation being that, under the assumption cyclic parallel assumption $D=0,$
  the identity $\frac{1}{4}|{\rm div} S|^2=3|\nabla\mathring{Ric}|^2$ holds. Hence, the Ricci tensor is parallel.

  For vacuum static spaces, we refer to the classification result of closed vacuum metrics with harmonic curvature due to Li \cite[Theorem 1.2]{Li2021} to prove Theorem \ref{thm3}. For CPE metrics, we directly apply the result of Chang, Hwang and Yun \cite[Theorem 1.1]{chang2012critical} to establish Theorem \ref{thm4}.
  \end{proof}
  \begin{Rmk}
    We point out that the integral identity \eqref{60}, obtained via the above procedure, is essentially equivalent to the integral identity derived in \cite{baltazar2018} by applying the Bochner formula
    \begin{align*}
        \frac{1}{2}{\rm div}((f+a)\nabla|Ric|^{2})
        &=(a+f)\Big(\frac{n-2}{n-1}|C_{ijk}|^{2}+|\nabla Ric|^{2}\Big)- (\frac{aR}{n-1}+nb)|\mathring{Ric}|^{2}\nonumber\\
        &\quad+(a+f)\Big(\frac{2}{n-1}R|\mathring{Ric}|^{2}+\frac{2n}{n-2}tr(\mathring{Ric}^{3})\Big)\nonumber\\
        &\quad-\frac{n-2}{n-1}W_{ijkl}\nabla_{l}fC_{ijk}-2(f+a)W_{ijkl}R_{ik}R_{jl}
    \end{align*}
    and then multiplying both sides by $f+a.$ The main difference, however, is that instead of directly dealing with the inner product term $\langle \iota_{\nabla f}W,C\rangle.$ 
  
  \end{Rmk}
  Finally, we briefly outline the proof of Theorem \ref{thm5}. It suffices to point out the differences between the compact-with-boundary Miao–Tam critical metrics and the closed manifolds considered above. In general, we consider the case of equation \eqref{eqn3} with $a=0$ and $b<0,$ where the manifold $(M,g)$ is compact with nonempty boundary. In this setting, the potential function $f$ is assumed to be positive in the interior and to vanish precisely on $\partial M$.

  \begin{proof}[Proof of Theorem \ref{thm5}]
     In this case, Lemma \ref{lemm2} does not hold in general; however, the other applications of integration by parts remain unaffected. Therefore, we need to recover the terms that were previously neglected by using equation \eqref{eqn3}.

    The identity in Lemma \ref{lemm3} should be replaced by the following
    \begin{align}
        \int_{\textit{M}}\nabla_{i}\mathring{R}_{jk}\mathring{R}_{ik}\nabla_{j}ff
        &=-\int_{\textit{M}}tr(\mathring{Ric}^3)f^2+\frac{1}{n(n-1)}\int_{\textit{M}}R|\mathring{Ric}|^2f^2\nonumber\\
        &\quad-\int_{\textit{M}}|\mathring{Ric}(\nabla f)|^2-b\int_{\textit{M}}|\mathring{Ric}|^2f.\nonumber
    \end{align}
    Equation \eqref{54} now takes the following form
    \begin{align}
        \int_{\textit{M}}\mathring{R}_{ik}S_{ijkl}\nabla_{j}\nabla_{l}ff
        &=\int_{\textit{M}}S_{ijkl}\mathring{R}_{ik}\mathring{R}_{jl}f^2-\frac{1}{n}\int_{\textit{M}}R|\mathring{Ric}|^2f^2+(n-1)b\int_{\textit{M}}|\mathring{Ric}|^2f.\nonumber
    \end{align}
    After updating \eqref{56} and \eqref{60}, and then adding them together, we obtain the following new integral identity
    \begin{align}\label{3.34}
        \int_{\textit{M}}|{\rm div}S|^2f^2&=4\int_{\textit{M}}|\nabla\mathring{Ric}|^2f^2+8\int_{\textit{M}}\nabla_{i}\mathring{R}_{jk}\mathring{R}_{ik}\nabla_{j}ff\nonumber\\
           &\quad+2\int_{\textit{M}}\langle \nabla|\mathring{Ric}|^2,\nabla f  \rangle f+4nb\int_{\textit{M}}|\mathring{Ric}|^2f,
    \end{align}
    which differs from the identity in Proposition \ref{prop5} by one additional term.

    Following the same procedure as before, we apply the curvature condition $D=0 $ to \eqref{3.34}, we arrive at 
    \begin{equation*}
        0=2\int_{\textit{M}}|\nabla\mathring{Ric}|^2f^2-nb\int_{\textit{M}}|\mathring{Ric}|^2f.
    \end{equation*}
    In the case where $b<0,$ we obtain directly a stronger result than previously: namely, the metric is Einstein. By the result of Miao–Tam \cite[Theorem 1.1]{Miao2011}, the proof is completed.

    \end{proof}

\section{Acknowledgement}
The author would like to express his sincere gratitude to his advisor, Jian Ye, for his continuous discussions, valuable suggestions, and strong support.

\end{document}